\documentclass[12pt]{article}
\usepackage{amsmath}
\usepackage{amssymb}
\usepackage{amsfonts}
\usepackage{graphicx,color}
\usepackage{mathrsfs}
\usepackage[dvipsnames]{xcolor}
\usepackage[all]{xy}
\usepackage{mathtools}
\usepackage{textgreek}
\usepackage{graphicx}
\usepackage{fancyhdr}
\usepackage{titlesec}
\usepackage[amsmath,amsthm,thmmarks]{ntheorem}
\usepackage{bm}
\usepackage{bbm}
\usepackage[intoc]{nomencl}
\usepackage{enumerate}

\usepackage[doi=false,isbn=false,url=false,eprint=false,sorting=nyt]{biblatex} 
\usepackage[]{appendix}

\usepackage[top=1in, bottom=1.25in, left=1.25in, right=1.25in]{geometry}
\usepackage[colorinlistoftodos, textwidth=4cm, shadow]{todonotes}
\usepackage[colorlinks, linkcolor=blue, citecolor=blue, urlcolor=blue]{hyperref}

\theoremstyle{plain}
\newtheorem{theorem}{Theorem}

\newtheorem{lemma}[theorem]{Lemma}

\newtheorem{proposition}[theorem]{Proposition}

\theoremstyle{definition}
\newtheorem{remark}[theorem]{Remark}

\newcommand{\diff}{\mathrm{d}}

\DeclarePairedDelimiterX\braket[2]{[\,}{\,]}{#1 \,\delimsize\vert\, #2}
\DeclarePairedDelimiterX\scalar[2]{\langle}{\rangle}{#1, #2}

\providecommand{\keywords}[1]
{
  \small	
  \textit{Keywords:} #1
}

\providecommand{\AMSid}[1]
{
  \small	
  \textit{2020 Mathematics Subject Classification:} #1
}

\begin{document}

\title{On the law of terminal value of additive martingales in a remarkable branching stable process}

\author{Hairuo Yang \footnote{Institute of Mathematics, University of Z\"{u}rich, Switzerland}}
\date{\today}

\maketitle
\begin{abstract}
    We give an explicit description of the law of terminal value $W$ of additive martingales in a remarkable branching stable process. We show that the right tail probability of the terminal value decays exponentially fast and the left tail probability follows that $-\log \mathbb{P}(W<x) \sim \frac{1}{2} (\log x)^2$ as $x \rightarrow 0+$. These are in sharp contrast with results in the literature such as Liu \cite{Liu00, Liu01} and Buraczewski \cite{B09}. We further show that the law of $W$ is self-decomposable, and therefore, possesses a unimodal density. We specify the asymptotic behavior at $0$ and at $+\infty$ of the latter. 
\end{abstract}

\AMSid{Primary 60G44; 60J80}

\section{Introduction} \label{se1}
We consider a branching random walk $(\mathbf{Z}_n)_{n \geqslant 0}$ on $\mathbb{R}^d$. At time $0$, the process starts with an ancestor that lives at the origin. At time $1$, the ancestor dies and simultaneously gives birth to children, which form the first generation of this process and whose positions are given by the atoms of some point process $\mathcal{Z}$. For each integer $n \geqslant 1$, every individual in the $n$th generation gives birth, independently of each other, to its own offspring that form the $(n+1)$th generation. The displacements of its children from this individual's position are given by the atoms of an independent copy of $\mathcal{Z}$. The atoms of point process $\mathbf{Z}_n$ are given by the locations of  individuals from the $n$th generation.

The Laplace transform of the intensity measure of $\mathbf{Z}_1$ is denoted by $m$, so that 
$$
m(\theta)=\mathbb{E}\left(  \int_{\mathbb{R}^d} e^{-\langle\theta, x\rangle}\mathbf{Z}_1(dx)  \right)
$$  
where $\theta \in \mathbb{R}^d$ and $\langle \cdot, \cdot \rangle$ denotes the scalar product on $\mathbb{R}^d$. Throughout this paper we consider only those $\theta$ for which $m(\theta) < \infty$. We also restrict our attention to supercritical branching random walks. In other words, we assume the expected number of children of a parent $\mathbb{E}(\# \mathbf{Z}_1)$ is strictly greater than $1$. The additive martingale $\left(W_n(\theta)\right)_{n \geqslant 0}$ defined by
$$
W_n(\theta)=m(\theta)^{-n} \int_{\mathbb{R}^d} e^{-\langle\theta, x\rangle}\mathbf{Z}_n(dx)
$$
is known to be a nonnegative martingale that converges almost surely to its \textit{terminal value} $W(\theta)$. The celebrated Biggins martingale convergence theorem in \cite{Biggins77} derived necessary and sufficient conditions for the terminal value to be non-degenerate. Specifically, 
under the assumption that $m(\theta)$ is finite, 
the terminal value $W(\theta)$ has expectation $1$, or equivalently, $\left(W_n(\theta)\right)_{n \geqslant 0}$ converges to $W(\theta)$ in $L^1$ if and only if 
\begin{align} \label{eq:nondegenerate}
\frac{\left\langle \theta,  \nabla m(\theta) \right\rangle}{m(\theta)} < \log m(\theta) \quad \textrm{and} \quad \mathbb{E} \left( W_1(\theta) \log^+\left(W_1(\theta)\right) \right) <\infty,
\end{align}
where $\nabla m(\theta)$ is the gradient of $m$ at $\theta$ and $\log^+(x) = \max( \log x, 0)$. 
This result was reproved under a slightly weaker condition in Lyons \cite{Lyons97} by a remarkable change of measures argument. The terminal values of additive martingales play a key role in the study of branching random walks. For example, Theorem 4 in Biggins \cite{Biggins92} described the spread of the $n$th generation in a branching random walk in terms of the terminal values of its additive martingales.

The law of the terminal value of an additive martingale is known explicitly in only very few special cases.  For example, in case of the standard Yule process, the terminal value of the corresponding additive martingale follows standard exponential distribution (cf. Remark \ref{rk:Yule} where we adapt the proof of our main result to give an easy proof of this statement). Most efforts have been devoted to gain insight on general properties of the law of terminal value of an additive martingale. Among those results are \cite{Liu00} regarding the existence of moments,  \cite{Liu01} regarding the support of distribution, \cite{BG79, Liu01} regarding the absolute continuity, \cite{Liu01} regarding the left tail probability. Precise estimates on the right tail probability of $W(\theta)$ have been derived by Liu \cite{Liu00} and Buraczewski \cite{B09}. More specifically, by Theorem 2.2 in \cite{Liu00}, if (\ref{eq:nondegenerate}) is satisfied, there exists some $\beta > 1$ such that  
\begin{align} \label{eq:Liu}
m(\beta\theta) = m^{\beta}(\theta) < \infty
\end{align}
and some further hypotheses are satisfied,
then $W(\theta)$ has power-law right tail such that
\begin{align} \label{eq:Liu2}
\mathbb{P}( W(\theta) > x ) \sim a {x^{-\beta}} \quad \textrm{as} \quad x \rightarrow \infty
\end{align}
for some $a \in (0, \infty)$. The study of terminal value $W(\theta)$ typically  starts with the observation that 
\begin{align} \label{smoothing}
    W(\theta) \stackrel{(\textrm{d})}{=} \frac{1}{m(\theta)} \sum_{ \textrm{$x$ atom of $\mathbf{Z}_1$} } e^{-\langle \theta, x \rangle} W^{(x)}(\theta)
\end{align}
where $\{W^{(x)}(\theta): x \textrm{ atom of }\mathbf{Z}_1\}$ is a family of i.i.d. copies of $W(\theta)$ that are independent of $\mathbf{Z}_1$ and $\stackrel{(\textrm{d})}{=}$ means equality in distribution. This implies that the law of $W(\theta)$ is a fixed point of a smoothing transform and properties of $W(\theta)$ can be obtained by studying the fixed points of smoothing transforms. Fix a sequence of $\mathbb{R}_+$-valued random variables $T_1, T_2, \ldots$ and a random variable $N$ that takes values in $\mathbb{Z}_+$. For any nonnegative random variable $X$ with its law denoted by $\mathcal{L}(X)$, let $X_1, X_2, \ldots$ be i.i.d copies of $X$ such that $(X_1, X_2, \ldots)$ is independent of $(N, T_1, T_2, \ldots)$. The smoothing transform associated to $(N, T_1, T_2, \ldots)$ is the mapping $\mathcal{L}(X) \mapsto \mathcal{L}\left(\sum_{ i \leq N} T_i X_i\right)$ and the fixed point of the above smoothing transform is the distribution of a random variable $X$ that satisfies relation
\begin{align*}
X \stackrel{(\textrm{d})}{=} \sum_{ i \leq N} T_i X_i.
\end{align*}
There is an extensive literature investigating fixed points of a smoothing transform. The problem of existence of solutions
was solved in Durrett and Liggett \cite{RL83} for deterministic $N$. Their results were generalised by Liu \cite{Liu98} to the case
where $N$ is random. Biggins and Kyprianou \cite{BK97, BK03} considered the problem of uniqueness and gave characterisation of the fixed points. In particular, under the assumption that (\ref{eq:nondegenerate}) holds, Theorem 1.5 in \cite{BK97} implies that there is only one nontrivial solution $W(\theta)$ to the equation (\ref{smoothing}) that has expectation $1$. The above equality (\ref{smoothing}) also implies that the law of $W(\theta)$ solves a more general fixed-point distributional equation that is of the form 
\begin{align} \label{eq:fixedpoint}
X \stackrel{(\textrm{d})}{=} AX + B
\end{align}
where $A, B$ are random variables independent of $X$.
We refer the reader to Buraczewski et al. \cite{BDM16} for further references on this equation and the more general stochastic recursion equations.

The purpose of this paper is to present an interesting and non-trivial case where fairly explicit results can be obtained regarding the laws of terminal values of additive martingales. We consider branching random walk $(\mathbf{Z}_n)_{n \geqslant 0}$ on $\mathbb{R}^2_+$ started from the origin where $\mathbf{Z}_1$ is a Poisson point process with intensity measure given by the Lebesgue measure on the upper-right quadrant. By the non-degenerate condition (\ref{eq:nondegenerate}), the associated additive martingale $\left(W_n(\theta)\right)_{n\geqslant 0}$ converges in $L^1$ to its terminal value $W(\theta)$ if and only if $\theta = (\theta_1, \theta_2) \in \mathbb{R}_+^2$ is such that $\log \theta_1 + \log \theta_2 < 2$. Note that $W((c, c^{-1}))$ is equal in distribution to $W\left( (1,1) \right)$ for any $c>0$ since Poisson point process $\mathbf{Z}_1$ is invariant under map $T$ defined on $\mathbb{R}^2_+$ by $T(x,  y) = (cx, c^{-1}y)$. In the following, we denote by $W = W\left((1,1) \right)$ the terminal value of additive martingale $\left( W_n\left((1,1)\right) \right)_{n\geqslant 0}$ associated to branching random walk $( \mathbf{Z}_n)_{n\geqslant 0}$.

Our main result derives an explicit formula for the law of $W$. We introduce the moment-generating function
\begin{align*}
    \phi(s) = \mathbb{E} (e^{sW}) 
\end{align*}
for each $s \in \mathbb{R}$ such that the above expectation is finite. We also introduce a constant
$$
r^{*} = \exp \left( \int_0^{\infty} \left({\left(2(e^u-u-1)\right)}^{-1/2}-u^{-1} \mathbbm{1}_{\{u<1\}}\right)\diff u \right)
$$ 
and a real-valued function $F: (1, \infty) \rightarrow (0, \infty)$ by
\begin{align*}
    F(y)=\int_{\log y}^{\infty} \frac{\diff u}{\sqrt{2(e^{u}-u-1)}}, \quad y>1.
\end{align*} One can show that the constant $r^* \in (0, \infty)$ by series expansion of $e^u - u -1.$
Furthermore, function $F$ is a strictly decreasing bijection. We denote its inverse function by $F^{-1}$. We adopt the usual notation $f(x) \sim g(x)$, $x \rightarrow x_0$ to denote the asymptotic equivalence that $\lim_{x \rightarrow x_0} f(x)/g(x) = 1$.

\begin{theorem} \label{theorem:simple}
(i) The moment-generating function $\phi$ of the terminal value $W$ is given by
\begin{align*}
\phi(r)=F^{-1}\left(\log r^{*}-\log r \right), \quad \forall\, 0<r<r^{*};
\end{align*}
the explosion point of $\phi$ is equal to $r^*$, that is,
\begin{align*}
 \sup \{r \in \mathbb{R} : \phi( r ) < \infty \} =  r^{*}.
\end{align*}
(ii) As a consequence, the right tail probability of $W$ decays exponentially with rate $r^*$, that is, 
$$
- \log \mathbb{P}(W>x) \sim r^{*}x \quad \textrm{as} \quad x \rightarrow \infty.
$$
(iii) The left tail probability of $W$ satisfies that
$$
- \log \mathbb{P}(W < x) \sim \frac{1}{2} (\log x)^2 \quad \textrm{as} \quad x \rightarrow 0\hspace{-2pt}+.
$$
\end{theorem}
We remark that the above tail probabilities are in sharp contrast with (\ref{eq:Liu2}) and the results obtained by Liu \cite{Liu00, Liu01} and Buraczewski \cite{B09}. In fact, for branching random walk $\{\mathbf{Z}_n\}_{n\geqslant 0}$, we have $m((\theta_1, \theta_2)) = (\theta_1 \theta_2)^{-1}$; thus, equation (\ref{eq:Liu}) fails for any $\beta>1$.

The proof of Theorem \ref{theorem:simple} replies on distributional equation (\ref{smoothing}) in its functional form 
\begin{align*} \label{smoothing:functional}
\phi(r)=\mathbb{E}\left( \prod_{\textrm{$(t,x)$ atom of $\mathbf{Z}_1$}} \phi(re^{-t-x})   \right).
\end{align*}
By construction, $\mathbf{Z}_1$ is a Poisson point process with explicitly known intensity measure. This enables us to transfer the above functional equation to a differential equation that has a unique solution. 

The branching random walk $(\mathbf{Z}_n)_{n \geqslant 0}$ considered in Theorem \ref{theorem:simple} is closely related to branching-stable processes. Recently introduced by Bertoin et al. \cite{BCM18}, branching-stable processes form a family of continuous-time branching processes that possess a self-similar property. For example, for every $t\geqslant 0$, let $\mathbf{S}_t$ be the point process on $\mathbb{R}_+$ given by
\begin{align*} \label{eq:selfsimilarity}
 \mathbf{S}_t ([0,x]) = \sum_{ n \geqslant 0}  \mathbf{Z}_n( [0,t] \times [0,x] ), \quad \forall\, x \geqslant 0.
\end{align*}
Then $(\mathbf{S}_t)_{t \geqslant 0}$ is a branching-stable process with stable index $-1$. More specifically, it is a continuous-time branching process that satisfies the distributional identity that
$$
(\mathbf{S}_{ct})_{t\geqslant 0} \stackrel{(\textrm{d})}{=}  (c^{-1}\mathbf{S}_t)_{t\geqslant0}, \quad \forall\, c>0,
$$
where $c^{-1}\, \mathbf{S}_t$ denotes the point process that re-scales the locations of atoms of $\mathbf{S}_t$ by constant $c^{-1}$. For any $\theta >0$, the additive martingale $(V_t(\theta))_{t \geqslant 0}$ associated to $(\mathbf{S}_t)_{t\geqslant 0}$ defined by
$$
V_t(\theta) = \int_0^{\infty} e^{-\theta x} \mathbf{S}_t (dx)\Big/\mathbb{E} \left(\int_0^{\infty} e^{-\theta x} \mathbf{S}_t (dx) \right) 
$$
converges almost surely to its terminal value $V(\theta)$. The convergence also holds in $L^2$ sense and $\mathbb{E}(V(\theta)) = 1$ for any $\theta>0$. This is due to a version of Biggins' additive martingale convergence theorem that was proved by Bertoin and Mallein \cite{BM18} for branching-stable processes and the more general branching Lévy processes. Their results were later completed for $L^p$ convergence by Iksanov and Mallein \cite{IM19}.  Furthermore, Corollary 3.7 in \cite{BCM18} obtained a precise description of asymptotic behaviors of a branching-stable process in terms of the terminal values of its additive martingales. For the specific process $(\mathbf{S}_t)_{t\geqslant 0}$, the last result states that, for any function $f: \mathbb{R} \rightarrow \mathbb{R}$ that is directly Riemann integrable with compact support, 
 $$
 \lim_{t \rightarrow \infty} \sqrt{t} e^{-2 \theta^{-1} t} \int_{0}^{\infty} f(y - \theta^{-2} t ) \mathbf{S}_t(dy) = \frac{V(\theta)}{\sqrt{4\pi \theta^{-3} }} \int_{-\infty}^{\infty} f(y) e^{ \theta y} dy
 $$
 where the limit is uniform for $\theta$ in compact subsets of $(0,\infty)$, almost surely. This also highlights the importance of studying the law of terminal value of an additive martingale.
 
 Due to the close relationship between $(\mathbf{S}_t)_{t\geqslant 0}$ and $(\mathbf{Z}_n)_{n\geqslant 0}$, terminal value $V(\theta)$ has the same law as $W$ for every $\theta>0$. In fact, we will show in Proposition \ref{prop:coincide} that, for a more general class of branching random walks and branching-stable processes, the terminal values of the associated additive martingales coincide almost surely. As a direct corollary, our main result regarding branching random walk $(\mathbf{Z}_n)_{n \geqslant 0}$ also gives an explicit description and asymptotic properties of the law of terminal values $V(\theta)$ of the additive martingales associated to branching-stable process $(\mathbf{S}_t)_{t \geqslant 0}$.

The stationarity of the Poisson point process $\mathbf{Z}_1$ also induces the self-decomposability of the law of $W$. Recall that, the law of some random variable $U$ is called self-decomposable if, for every $s>0$, there exists some random variable $U_s$ independent of $U$ such that
$$
U \stackrel{(\textrm{d})}{=} e^{-s} U + U_s.
$$
We refer the reader to Sato \cite{Sato1999} for detailed information on self-decomposable laws and remark that in Pakes \cite{P20}, self-decomposability also appears in the study of martingale limits of continuous-state branching processes. Note that the above distributional equation is a special case of the fixed-point distributional equation (\ref{eq:fixedpoint}). We will show in Proposition \ref{prop:sd} that, for general branching-stable processes and the corresponding branching random walks, the laws of the terminal values of their additive martingales are self-decomposable. 
As a consequence of the self-decomposability, the law of $W$ is absolutely continuous with unimodal density $f$. Recall that a probability density function $h: \mathbb{R} \rightarrow \mathbb{R}$ is called unimodal if there exists some real number $x_0$ such that $h$ is non-decreasing on $(-\infty, x_0)$ and non-increasing on $(x_0, +\infty)$. This allows us to upgrade results from Theorem \ref{theorem:simple} to the following that gives a precise description of asymptotic behaviors of the density $f$.
\begin{theorem} \label{prop:density}
The law of $W$ is self-decomposable and admits unimodal density function $f$ such that
\begin{align*}
- \log f(x) \sim r^{*}x \quad \textrm{as} \quad x \rightarrow \infty \quad \textrm{and}
\quad - \log f(x) \sim \frac{1}{2} (\log x)^2 \quad \textrm{as} \quad  x \rightarrow 0\hspace{-3pt}+.
\end{align*}
\end{theorem}

The remaining sections of this paper are organised as follows. Section \ref{se2} starts with Lemma \ref{lemma:finiteexponentialmoment} that shows the terminal values of additive martingales of some branching random walks admit some finite exponential moments. Then we present the proof of Theorem \ref{theorem:simple}. In Section \ref{se3}, we first review some background on branching-stable processes. Then we show in Proposition \ref{prop:coincide} that the terminal value of additive martingales associated to some branching-stable process is equal to that associated to some branching random walk. After that, we show in Proposition \ref{prop:sd} that the terminal values of additive martingales associated to some branching stable processes follow self-decomposable laws. Finally, we obtain asymptotic behaviors of the density of the terminal value $W$ which is Theorem \ref{prop:density}.

\section{Proofs} \label{se2}

We begin with a statement for slightly more general branching random walks. For each $\alpha>0$, let $( \mathbf{Z}^{(\alpha)}_n )_{n \geqslant 0}$ be a branching random walk on $\mathbb{R}_+^2$ started from the origin such that $\mathbf{Z}_1^{(\alpha)}$ is a Poisson point process on $\mathbb{R}^2_+$ with intensity measure given by 
$$
\Gamma{(\alpha)}^{-1} \mathbf{1}_{ \{x >0, t >0 \} } x^{\alpha-1}dxdt
$$
where $dxdt$ denotes the Lebesgue measure on $\mathbb{R}^2_+$ and $\Gamma$ denotes the usual Gamma function. Note that taking $\alpha = 1$ would reduce our branching random walk $(\mathbf{Z}_n^{(\alpha)})_{n \geqslant 0}$ to the case $(\mathbf{Z}_n)_{n \geqslant 0}$ considered in Theorem \ref{theorem:simple}. Let $W^{(\alpha)}$ denote the terminal value of the additive martingale $\{W_n^{(\alpha)}\}_{n\geqslant 0}$ defined by
$$
W_n^{(\alpha)} = \int_0^{\infty}\int_0^{\infty} e^{-t-x} \mathbf{Z}^{(\alpha)}_n(dt, dx).
$$
By (\ref{eq:nondegenerate}), we know that $W_n^{(\alpha)}$ converges almost surely and in $L^1$ to its terminal value $W^{(\alpha)}$ and $\mathbb{E}\left(W^{(\alpha)}\right) = 1$. 

\begin{lemma}\label{lemma:finiteexponentialmoment}
The moment-generating function $\phi^{(\alpha)}$ of $W^{(\alpha)}$ is given by the unique solution of equation 
\begin{align} \label{eq3}
    \log \phi^{(\alpha)}(r)=\Gamma(\alpha)^{-1} \int_{\mathbb{R}_+^2} \left(\phi^{(\alpha)}(re^{-t-x})-1\right) x^{\alpha-1}  \diff x\diff t , \quad r < r^{(\alpha)},
\end{align}
where the explosion point $r^{(\alpha)}$ of $\phi^{(\alpha)}$ defined by
$$
r^{(\alpha)}= \sup \{ r \in \mathbb{R}: \phi^{(\alpha)}(r) < \infty \}
$$
is a finite positive number. Furthermore, the $k$-th moment $\mu_k^{(\alpha)}$ of $W^{(\alpha)}$ satisfies a recursive relation
\begin{align*}
\mu_k^{(\alpha)}=\left(1-\frac{1}{k!k^{1+\alpha}}\right)^{-1} \sum_{m=1}^{k-1}{k-1 \choose m-1} \frac{\mu^{(\alpha)}_m \mu^{(\alpha)}_{k-m}}{m!m^{1+\alpha}}, \quad \mu^{(\alpha)}_1=1.
\end{align*}
that characterises the distribution of $W^{(\alpha)}$.
\end{lemma}

\begin{proof}[Proof of Lemma \ref{lemma:finiteexponentialmoment}]
By the explicit intensity measure of Poisson point process $\mathbf{Z}_1^{(\alpha)}$,  the moment-generating function of $W^{(\alpha)}_1$ can be computed as follows
\begin{align*}
    \mathbb{E}(\exp(r W^{(\alpha)}_1 ))=\exp\left( \int_0^{\infty} \int_0^{\infty} \left(e^{re^{-t}e^{-x}}-1\right) \Gamma(\alpha)^{-1} x^{\alpha-1} \,dtdx \right).
\end{align*}
 For every $r >0$, using series expansion 
 $$
 e^{re^{-t-x}} - 1 =\sum_{k \geqslant 1} \frac{1}{k!} r^k e^{-kt}e^{-kx},
 $$
 we can calculate the double integral in the above expression and obtain that
\begin{align*}
\mathbb{E}(\exp(r W^{(\alpha)}_1 )) = \exp\left( \sum_{k \geqslant 1} \frac{1}{k! \,  k^{1+\alpha}} r^k \right). 
\end{align*}
By Theorem 2.3 in Liu \cite{Liu00}, we conclude that the terminal $W^{(\alpha)}$ admits some finite exponential moment and the explosion point $r^{(\alpha)}$ of  $\phi^{(\alpha)}$ is a finite positive number.

Recall that the law of $W^{(\alpha)}$ is a fixed point of a smoothing transform, i.e.,
\begin{align*}
    W^{(\alpha)} \stackrel{(\textrm{d})}{=} \sum_{ \textrm{$(t,x)$ atom of $\mathbf{Z}^{(\alpha)}_1$} } e^{-t-x} W^{(\alpha,t, x)}
\end{align*}
where $\{W^{(\alpha, t, x)}: (x, t) \textrm{ atom of }\mathbf{Z}^{(\alpha)}_1\}$ is a family of i.i.d. copies of $W^{(\alpha)}$ that are independent with $\mathbf{Z}^{(\alpha)}_1$. In terms of the moment-generating function, we have
\begin{align*} \label{smoothingtransformoriginal}
 \phi^{(\alpha)}(r)=\mathbb{E}\left( \prod_{(t,x) \textrm{ atom of }\mathbf{Z}_1^{(\alpha)}}  \phi^{(\alpha)}(re^{-t-x})  \right).
\end{align*}
Using again the fact that $\mathbf{Z}_1^{(\alpha)}$ is a Poisson point process with explicitly given intensity measure, we obtain
\begin{align}  \label{smoothingtransform}
 \log \phi^{(\alpha)}(r)= \Gamma(\alpha)^{-1} \int_0^{\infty} \int_0^{\infty} \left(\phi^{(\alpha)}(re^{-t-x})-1\right) \,\diff t  x^{\alpha-1}\diff x.
\end{align} 
Since the moment-generating function $\phi^{(\alpha)}$ is analytic on some vertical strip containing the origin and $\phi^{(\alpha)}(0)=1$, we can write
\begin{align} \label{cumulant1}
\phi^{(\alpha)}(r)=\sum_{k \geqslant 1}  \frac{1}{k!}\mu^{(\alpha)}_k r^k +1
\end{align}
where $\mu^{(\alpha)}_k$ is the $k$-th moment of $W^{(\alpha)}$.
Combining equations (\ref{smoothingtransform}) and (\ref{cumulant1}) and evaluating the double integral, it follows that
\begin{align}  \label{cumulant2}
 \log \phi^{(\alpha)}(r) =\sum_{k \geqslant 1}  \frac{1}{k! k^{1+\alpha}} \mu^{(\alpha)}_k r^k.
\end{align}
Taking derivatives on both sides of the above, we can express the cumulant $c_k^{(\alpha)}$ of $W^{(\alpha)}$ in terms of moment $\mu_k^{(\alpha)}$ as
$$
c_k^{(\alpha)}=\frac{1}{k!k^{1+\alpha}} \mu^{(\alpha)}_k.
$$
Recall that the following elementary formula relates cumulants and moments of a random variable
$$
\mu_k^{(\alpha)}=\sum_{m=1}^{k-1}{k-1 \choose m-1} c^{(\alpha)}_m \mu^{(\alpha)}_{k-m}+c^{(\alpha)}_k.
$$ 
Hence, we can recursively compute the moments of $W^{(\alpha)}$ by 
\begin{align*}
\mu_k^{(\alpha)}=\left(1-\frac{1}{k!k^{1+\alpha}}\right)^{-1} \sum_{m=1}^{k-1}{k-1 \choose m-1} \frac{\mu^{(\alpha)}_m \mu^{(\alpha)}_{k-m}}{m!m^{1+\alpha}}, \quad \mu^{(\alpha)}_1=1.
\end{align*}
Furthermore, the above characterise the distribution of $W^{(\alpha)}$. This completes the proof of Lemma \ref{lemma:finiteexponentialmoment}.
\end{proof}

Now we have all the ingredients needed to prove the three claims in Theorem \ref{theorem:simple}.

\begin{proof}[Proof of Theorem \ref{theorem:simple} (i)] We take $\alpha=1$. Recall that by definition we have $W^{(1)} = W$ and $\phi^{(1)} = \phi$. In this case, we are able to express the moment-generating function $\phi$ of the terminal value $W$ in an explicit way. To ease our notation, instead of the moment-generating function, we work with the cumulant-generating function $\psi$ of $W$ where
$$
\psi(r) = \log \phi(r)=\log\left( \mathbb{E}\left( \exp(rW)  \right) \right)
$$
for $r < r^{(1)}$ where 
$$
r^{(1)} = \inf\{ s \in \mathbb{R}: \phi(s) < \infty\} = \inf\{ s \in \mathbb{R}: \psi(s) < \infty\}.
$$
Recall that we have shown in the proof of Lemma \ref{lemma:finiteexponentialmoment} that $r^{(1)}$ is a finite positive number.
Taking $\alpha=1$ in (\ref{cumulant1}) and $(\ref{cumulant2})$ gives us that
$$
e^{\psi(r)}=\sum_{k \geqslant 1} \frac{c^{(1)}_k r^k}{k!}+1 \quad \textrm{and}\quad \psi(r)=\sum_{k \geqslant 1} \frac{c^{(1)}_kr^k}{k!k^2}.
$$
This implies that
$$
r (r\psi'(r))' =e^{\psi(r)}-1.
$$
We multiply both sides of the above by $\psi'(r)$ and then integrate them from $0$ to $r$. Noting that $\psi(0)=0$, we obtain an ordinary differential equation
\begin{align} \label{separationofvariables2}
r^2 \left(\psi'(r)\right)^2 =2 \left(e^{\psi(r)}-\psi(r)-1\right), \quad r<r^{(1)}.
\end{align}
By the method of separation of variables, we have that
\begin{align} \label{separationofvariables}
H(\psi(r))=\log r^{(1)}-\log r, \quad 0<r<r^{(1)}
\end{align}
where, for every $y>0$, we let
$$
H(y)=\int_y^{\infty} \frac{\diff u}{\sqrt{2(e^u-u-1)}}.
$$
Noting that $F(y) = H(\log y)$ gives us the desired expression of $\psi$ and $\phi$. To complete the proof of the part (i) of Theorem \ref{theorem:simple}, we estimate function $H$ above around $0$ to get an explicit expression for explosion point $r^{(1)}$. Recall that, as $r \rightarrow 0$,
$$
\psi(r)=\mathbb{E}(W)\,r+o(r)=r(1+o(1)).
$$
Concerning the behavior of $H$ at $0$, observe that, for every $0<y<1$, 
\begin{align*}
H(y)&=-\log y+\int_y^{\infty}  \left(\frac{1}{\sqrt{2(e^u-u-1)}}-\frac{1}{u} \mathbbm{1}_{\{u<1\}}\right)\diff u.
\end{align*}
It follows from (\ref{separationofvariables}) and the above asymptotic expansion of $H$ that, as $r\rightarrow 0$,
\begin{align*}
\log r^{(1)} &=-\log (1+o(1))+\int_{r(1+o(1))}^{\infty} \left(\frac{1}{\sqrt{2(e^u-u-1)}}-\frac{1}{u} \mathbbm{1}_{\{u<1\}}\right)\diff u.
\end{align*}
Sending $r \rightarrow 0$ gives us an explicit expression of $r^{(1)}$ by
$$
\log r^{(1)}=\int_0^{\infty} \left(\frac{1}{\sqrt{2(e^u-u-1)}}-\frac{1}{u} \mathbbm{1}_{\{u<1\}}\right)\diff u.
$$
Note that the above integral converges absolutely since there exist some $C \geqslant 0$ such that, for all $u$ small enough,
\begin{align*}
\left|\frac{1}{\sqrt{2(e^u-u-1)}}-\frac{1}{u} \right| \leq  \frac{\sqrt{2(e^u-u-1)}-u}{u\sqrt{2(e^u-u-1)}} \leq    \frac{\sqrt{2\sum_{n \geqslant 3} u^n/n!}}{u^2} \leq C u^{-1/2}.
\end{align*}
This shows that the explosion point $r^{(1)}$ is equal to $r^*$ defined in Theorem \ref{theorem:simple}.
\end{proof}

\begin{proof}[Proof of Theorem \ref{theorem:simple} (ii)]
Now we derive the asymptotic behaviors of $\psi$ and $\phi$ as $r \uparrow r^{*}$. Note that, by (\ref{separationofvariables}),
$$
\log r^{*}-\log r \geqslant \int_{\psi(r)}^{\infty} \frac{1}{\sqrt{2 e^{u}} }  \diff u =\sqrt{2} e^{-\psi(r)/2}. 
$$
Fix $\delta>0$. Find $u(\delta)>0$ such that, whenever $u>u(\delta)$, we have $e^{u}-u-1 \geqslant (1-\delta) e^{u}$. Thus, for $r$ close enough to $r^{*}$, it holds that
$$
\log r^{*}-\log r \leq \int_{\psi(r)}^{\infty} \frac{1}{\sqrt{2(1-\delta) e^{u}}}  \diff u =\sqrt{\frac{2}{1-\delta}} e^{-\psi(r)/2}. 
$$
We turn our attention back to the moment-generating function $\phi$. The above two inequalities give us that, for every $\delta>0$, there is some $r(\delta)>0$ such that, for every $r>r(\delta)$, 
$$
2 (\log r^{*}-\log r)^{-2} \leq \phi(r) \leq \frac{2}{1-\delta} (\log r^{*}-\log r)^{-2}.
$$
Therefore, the behavior of $\phi$ around its point of explosion $r^*$ satisfies that
$$
\phi(r^{*}-\varepsilon) \sim \frac{2(r^{*})^2}{\varepsilon^2}, \quad \varepsilon \rightarrow 0+.
$$
By Criterion 1(ii) in \cite{BF08}, we can conclude that the right tail probability of $W$ decays exponentially with rate $r^{*}$, i.e.,
$$
-\log \mathbb{P}(W>x) \sim r^{*} x, \quad x \rightarrow \infty.
$$
For the reader's convenience, we briefly recall their arguments and adapt them to our case below. We define a new measure $U$ on $[0, \infty)$ such that 
$$
U(dx) = e^{r^* x} \mathbb{P}(W \in dx).
$$
Let $\hat U$ denote the Laplace transform of $U$. Then the behavior of $\phi$ at $r^*$ entails that $\hat U$ is regularly varying at $0$ with index $-2$. By Karamata’s Tauberian theorem or Theorem 1.7.1 of Bingham et al. \cite{BGT87}, the function $h(x)=U([0, x))$
is regularly varying with index $2$. 

Assume $h$ is smoothly varying. Then the derivative $h'$ is smoothly varying with index $1$ and there exists some slowly varying function $l$ such that $h'(y) = y\, l(y).$ By a Tauberian theorem of exponential type or Theorem 4.12.10 in Bingham et al. \cite{BGT87}, it follows that, as $x \rightarrow \infty$,
\begin{align} \label{eq:smoothvarying}
- \log \mathbb{P}(W > x) &= -\log \int_x^{\infty} \exp(-r^* y) U(dy) \nonumber\\
&= - \log \int_x^{\infty} \exp\Big( -r^* y + \log y + \log l(y) \Big)  dy \nonumber\\
& \sim r^*x.
\end{align}
The general case where $h$ is regularly varying can also be treated. In this case, we apply smooth variation theorem to find smoothly varying functions $h_-$ and $h_+$ such that $h_- \leq h \leq h_+$. The detailed arguments are given in \cite{BF08}.
\end{proof}

\begin{proof}[Proof of Theorem \ref{theorem:simple} (iii)]
Concerning the left tail probability, we examine the behavior of the Laplace transform $\varphi(r)=\mathbb{E}(e^{-rW})$ as $r \rightarrow \infty$. Recall that from (\ref{separationofvariables2}), we can derive that
\begin{align} \label{eq:100}
\frac{\varphi'(r)}{\varphi(r)\sqrt{\varphi(r)-\log \varphi(r)-1}}=-\frac{\sqrt{2}}{r}, \quad r>0.
\end{align}
For $0<y<\varphi(1)$, let $$G(y)=\int^{\varphi(1)}_y \frac{du}{u\sqrt{u-\log u-1}}.$$ 
Integrating both sides of (\ref{eq:100}) from $1$ to $r$ gives
\begin{align} \label{eq:behavioratinfinity}
G(\varphi(r))=\sqrt{2} \log r, \quad r>1.
\end{align}
Let 
$$
\tilde C=-2\sqrt{-\log \varphi(1)}-\int_0^{\varphi(1)}  \left(\frac{1}{u\sqrt{-\log u}}-\frac{1}{u\sqrt{u-\log u-1}} \right)du.
$$
Note that the above integral is finite since there exists some $C_1 \geqslant 0$ such that, for all $u$ small enough,
\begin{align*}
\left| \frac{1}{u\sqrt{-\log u}}-\frac{1}{u\sqrt{u-\log u-1}} \right| & \leq \frac{\sqrt{- \log u } - \sqrt{ u- \log u - 1}}{ u  \sqrt{ - (u - \log u -1) \log u}  } \leq \frac{C_1}{ u \log u}.
\end{align*}
We can describe the behavior of function $G$ around $0$ by 
\begin{align*}
G(y)&=\int_y^{\varphi(1)}  \frac{du}{u\sqrt{-\log u}}- \int_y^{\varphi(1)}  \left(\frac{1}{u\sqrt{-\log u}}-\frac{1}{u\sqrt{u-\log u-1}} \right)du  \\
&=2\sqrt{-\log y}-2\sqrt{-\log \varphi(1)}-\int_y^{\varphi(1)}  \left(\frac{1}{u\sqrt{-\log u}}-\frac{1}{u\sqrt{u-\log u-1}} \right)du\\
&=2\sqrt{-\log y}+\tilde C+\int_0^{y}  \left(\frac{1}{u\sqrt{-\log u}}-\frac{1}{u\sqrt{u-\log u-1}} \right)du\\
&=2\sqrt{-\log y}+\tilde C+o(1).
\end{align*}
Replacing $y$ by $\varphi(r)$ in the above, sending $r\rightarrow \infty$ and comparing with (\ref{eq:behavioratinfinity}), we obtain that
$$
- \log \varphi(r) \sim \frac{1}{2} (\log r)^2, \quad r \rightarrow \infty.
$$
By Corollary 1 of Bingham et al. \cite{BT75} and Remark \ref{remark:infinitelydivisable}, we can conclude that left tail probability satisfies that 
$$- \log \mathbb{P}(W < x) \sim \frac{1}{2} (\log x)^2, \quad x\rightarrow 0\hspace{-2pt}+. $$  
This completes the proof of Theorem \ref{theorem:simple}.
\end{proof}

\section{Some background on branching-stable processes} \label{se3}

\subsection{Construction of branching-stable processes} \label{se2.1}

In this subsection, we recall the construction of branching-stable processes introduced in \cite{BCM18}. Beware that here we use a slightly different notation. Let 
$\mathcal{M}^*$ be the space of locally finite point measures on $\mathbb{R}_+$ with no atom at $0$ and $
\mathcal{M}^1$ be the space of point measures in $\mathcal{M}^*$ with left-most atom located at $1$. For any  $\mathbf{x}$ in $\mathcal{M}^*$, let $(x_i)_{i \geqslant 1}$ denote  the sequence of atoms of $\mathbf{x}$ ranked in non-decreasing order and repeated according to their multiplicities. We use notation $\mathbf{x}$ and $(x_i)_{i\geqslant 1}$ for a point measure interchangeably. 

Fix $\alpha>0$ and finite measure $\lambda$ that is supported on $\mathcal{M}^1$. We work under the assumption that
\begin{align} \label{nondegenerate}
\int_{\mathcal{M}^1} \sum_{i \geq 1} x_i^{-\alpha} \,\lambda(d\mathbf{x})=\Gamma(\alpha)^{-1}
\end{align}
where $\Gamma$ denotes the usual gamma function. This is only a convenient normalization made to unburden the notation. Let sigma-finite measure $\Lambda^*$ on $\mathcal{M}^*$ be the image of $u^{\alpha-1}du \otimes \lambda(d\mathbf{x})$ under the map $(u,\mathbf{x}) \mapsto u\mathbf{x}:=(ux_i)_{i\geqslant 1}$, i.e., $\Lambda^*$ satisfies that
$$
\int_{\mathcal{M}^*} F(\mathbf{x}) \, \Lambda^*(d\mathbf{x})=\int_0^{\infty}  u^{\alpha-1} \int_{\mathcal{M^*}} F(u\mathbf{x}) \lambda(d\mathbf{x}) du
$$
for every measurable functional $F: \mathcal{M}^* \rightarrow \mathbb{R}_+.$ 

Assume $\mathbf{N}$ is a Poisson point process on space $\mathbb{R}_+ \times \mathcal{M}^*$ with intensity measure given by $dt \otimes \Lambda^*(d\mathbf{x})$. Regarding each atom $(t,\mathbf{x})$ of $\mathbf{N}$ as a sequence of atoms $(t,x_1), (t,x_2), \ldots$ on the fiber $\{t\} \times (0,\infty]$ induces a point process $\mathbf{Z}_1^{(\alpha, \lambda)}$ on the upper-right quadrant $\mathbb{R}_+^2$ that satisfies 
\begin{align} \label{eq:Poisson}
\int_0^{\infty} f(x)\, \mathbf{Z}_1^{(\alpha, \lambda)}(dx) =\int_{\mathbb{R}_+ \times  \mathcal{M}^*}\sum_{i \geqslant 1} f(t,x_i) \mathbf{N}(dt, d\mathbf{x})
\end{align}
for any generic function $f: \mathbb{R}_+^2 \rightarrow \mathbb{R}_+$.

We regard $\mathbf{Z}_1^{(\alpha, \lambda)}$ as the first step of a branching random walk $(\mathbf{Z}_n^{(\alpha, \lambda)})_{n \geqslant 0}$ on $\mathbb{R}_+^2$ started from a single atom at the origin. Consider the superposition $\bigsqcup_{n \geqslant 0} \mathbf{Z}_n^{(\alpha, \lambda)}$ that consists of all atoms in the branching random walk $(\mathbf{Z}_n^{(\alpha, \lambda)})_{n \geqslant 0}$. Let $\mathbf{S}_t^{(\alpha, \lambda)}$ be the point process whose atoms are given by the projection of atoms of $\bigsqcup_{n \geqslant 0} \mathbf{Z}_n^{(\alpha, \lambda)}$ restricted to the strip $[0,t] \times \mathbb{R}_+$. In other words, $\mathbf{S}_t^{(\alpha, \lambda)}$ is a point process such that, for every measurable $g: \mathbb{R}_+ \rightarrow \mathbb{R}_+$,
$$
\int_0^{\infty} g(x)\, \mathbf{S}_t^{(\alpha, \lambda)}(dx)=\sum_{n\geqslant 0}\int_{0}^t \int_{ \mathbb{R}_+} g(x) \, \mathbf{Z}_n^{(\alpha, \lambda)}(dt, dx).
$$
By Theorem 2.4 in \cite{BCM18}, the process $(\mathbf{S}^{(\alpha, \lambda)}_t)_{t\geqslant 0}$ is a branching-stable process with scaling exponent $-\alpha$, i.e., $(\mathbf{S}^{(\alpha, \lambda)}_t)_{t\geqslant 0}$ is a continuous-time branching process such that
\begin{align*} 
(\mathbf{S}^{(\alpha, \lambda)}_{ct})_{t\geqslant 0} \stackrel{(\textrm{d})}{=}  (c^{-\alpha}\mathbf{S}^{(\alpha, \lambda)}_t)_{t\geqslant0}, \quad \forall\, c>0,
\end{align*}
where $c^{-\alpha}\,\mathbf{S}^{(\alpha, \lambda)}_t$ denotes the point process that re-scales the locations of atoms of $\mathbf{S}^{(\alpha, \lambda)}_t$ by constant $c^{-\alpha}$. The same theorem also shows that every branching-stable process can be constructed in the above way. In particular, $\alpha$ and $\lambda$ characterise the law of $(\mathbf{S}^{(\alpha, \lambda)}_t)_{t\geqslant 0}$. Let $\mathbf{1}$ denote the point measure $\delta_1$ on $\mathbb{R}$ that has exactly one atom at $1$. Note that in case $\alpha = 1$ and $\lambda = \delta_{\mathbf{1}}$ reduces branching-stable process $(\mathbf{S}^{(\alpha, \lambda)}_t)_{t\geqslant 0}$ to  $(\mathbf{S}_t)_{t\geqslant 0}$ and branching random walk $(\mathbf{Z}^{(\alpha, \lambda)}_n)_{n\geqslant 0}$ to  $(\mathbf{Z}_n)_{ n \geqslant 0}$ that are discussed in the Introduction.

\subsection{Terminal values of additive martingales coincide} \label{se2.2}
In this subsection, we consider additive martingales associated to branching-stable process $(\mathbf{S}_t^{(\alpha, \lambda)})_{t \geqslant 0}$ and branching random walks $(\mathbf{Z}_n^{(\alpha, \lambda)})_{n \geqslant 0}$ and show that the corresponding terminal values of additive martingales coincide almost surely.

Assume that (\ref{nondegenerate}) holds. Proposition 3.2(ii) in \cite{BCM18} states that, for every $\theta>0$,
$$
\mathbb{E}\left( \int_0^{\infty} e^{-\theta x} \mathbf{S}^{(\alpha, \lambda)}_t(dx)\right)=e^{t\theta^{-\alpha}}<\infty. 
$$
Consider the additive martingale $(V_t^{(\alpha, \lambda)}(\theta))_{t\geqslant 0}$ associated to $(\mathbf{S}^{(\alpha, \lambda)}_t)_{t \geqslant 0}$ defined by
$$
V_t^{(\alpha, \lambda)}(\theta)= e^{-t\theta^{-\alpha}}\int_0^{\infty} e^{-\theta x} \mathbf{S}_t^{(\alpha, \lambda)}(dx)
$$
and its terminal value
$$
V^{(\alpha, \lambda)}(\theta)=\lim_{ t \rightarrow \infty} V_t^{(\alpha, \lambda)}(\theta).
$$
In the rest of this paper, we also assume the following integrability condition on $\lambda$ that, for some $p \in (1, 2]$, 
\begin{align} \label{eq:pthmoment}
\int_0^{\infty} u^{\alpha-1} \int_{\mathcal{M}^1} \left( \sum_{i \geqslant 1} e^{-ux_i} \right)^p \lambda(d\mathbf{x}) du<\infty.
\end{align}
By Proposition 3.5 in \cite{BCM18}, the additive martingale $(V_t^{(\alpha, \lambda)}(\theta))_{t\geqslant 0}$ also converges in $L^p$.

In a different direction, note that, for every $a,b>0$, 
$$
\mathbb{E}\left( \int_0^{\infty} \int_0^{\infty} e^{-at-bx} \mathbf{Z}_1^{(\alpha, \lambda)}(dt, dx)  \right) =  (ab^{\alpha})^{-1}.
$$
Consider additive martingale $(W_n^{(\alpha, \lambda)}(a,b))_{n \geqslant 0}$ defined by
$$
W_n^{(\alpha, \lambda)}(a,b) = (a b^{\alpha})^n \int_0^{\infty} \int_0^{\infty} e^{-at-bx} \mathbf{Z}_n^{(\alpha, \lambda)}(dt, dx) 
$$
and its terminal value 
$$
W^{(\alpha, \lambda)}(a,b) = \lim_{n \rightarrow \infty} W_n^{(\alpha, \lambda)}(a,b). 
$$
Let $g(t, \mathbf{x}) = a b^{\alpha} e^{-at} \sum_{i \geqslant 1} e^{-bx_i}$. Then 
\begin{align*}
W_1^{(\alpha, \lambda)}(a,b) = \int g(t, \mathbf{x}) \mathbf{N}(dt, d\mathbf{x})
\end{align*}
where $\mathbf{N}$ is a Poisson point process with intensity measure $dt \otimes \Lambda^*(dx)$. Note that $W_1^{(\alpha, \lambda)}(a,b) \in L^p$ if and only if $g \in L^p( dt \otimes \Lambda^*(dx))$, which is readily seen to be equivalent to (\ref{eq:pthmoment}). 

I learnt from Jean Bertoin the following result that gives a direct connection between terminal values of additive martingales of $(\mathbf{S}^{(\alpha, \lambda)}_t)_{t\geqslant 0}$ and $(\mathbf{Z}^{(\alpha, \lambda)}_n)_{n\geqslant 0}$.
\begin{proposition}\label{prop:coincide}
For every $\theta > 0$, it holds almost surely that
$$
W^{(\alpha, \lambda)}(\theta^{-\alpha}, \theta)=V^{(\alpha, \lambda)}(\theta).
$$

\end{proposition}

\begin{proof}
By the construction of $(\mathbf{S}^{(\alpha, \lambda)}_t)_{t\geqslant 0}$ in terms of $(\mathbf{Z}^{(\alpha, \lambda)}_n)_{n \geqslant 0}$ discussed in Section 3.1, 
$$
V^{(\alpha, \lambda)}_t(\theta) = e^{-t \theta^{-\alpha}} \int_0^{\infty} e^{-\theta x} \mathbf{S}^{(\alpha, \lambda)}_t(dx) = \sum_{n \geqslant 0} e^{-t \theta^{-\alpha}} \int_0^t \int_0^{\infty} e^{-\theta x}\mathbf{Z}^{(\alpha, \lambda)}_n(ds, dx).
$$
By Fubini's theorem, for every $q>0$, we have
\begin{align} \label{eq:coincide}
    \int_0^{\infty} q e^{-qt} V^{(\alpha, \lambda)}_t(\theta) dt=&  \sum_{n \geqslant 0}  \int_0^{\infty} q e^{-(q+\theta^{-\alpha})t}\int_0^{t} \int_0^{\infty} e^{- \theta x} \mathbf{Z}^{(\alpha, \lambda)}_n(ds,dx)dt \nonumber\\
    =& \sum_{n \geqslant 0} \int_0^{\infty}\int_0^{\infty} \frac{q}{q+\theta^{-\alpha}}  e^{-(q+\theta^{-\alpha})s-\theta x}\mathbf{Z}^{(\alpha, \lambda)}_n(ds,dx) \nonumber\\
    =& \sum_{n \geqslant 0} \frac{q}{q+\theta^{-\alpha}} (q\theta^{-\alpha}+1)^{-n}  W_n^{(\alpha, \lambda)}(q+\theta^{-\alpha}, \theta). 
\end{align}
Since $(V^{(\alpha, \lambda)}_t(\theta))_{t\geqslant 0}$ is bounded in $L^p$,  the left hand side converges almost surely to $V^{(\alpha, \lambda)}(\theta)$ as $q \rightarrow 0$. Since $p>1$ and $\alpha>0$, we can find a compact neighbourhood $K \subset \mathbb{R}_+ \times \mathbb{R}_+$ of $(\theta^{-\alpha},\theta)$ such that
\begin{align}\label{eq:Biggins}
(ab^{\alpha})^{p-1} < p^{1+\alpha} \quad \textrm{for all }(a,b) \in K.
\end{align}
Recall that (\ref{eq:pthmoment}) ensures that  $W^{(\alpha, \lambda)}_1(a,b)$ has finite $p$-th moment. On the other hand, (\ref{eq:Biggins}) rephrases the condition
(2.2) in Theorem 1 of Biggins \cite{Biggins92}. We can now apply Theorem 2 of \cite{Biggins92} to conclude that the convergence of martingale $(W^{(\alpha, \lambda)}_n(a,b))_{n\geqslant 0}$ holds uniformly in $K$. This gives us that the right hand side of equation (\ref{eq:coincide}) converges a.s. to the terminal value $W$ as $q\rightarrow 0$, which completes the proof.
\end{proof}

\subsection{Self-decomposability} \label{se2.3}
Recall that the law of point process $\mathbf{Z}^{(\alpha, \lambda)}_1$ on $\mathbb{R}_+^2$ is invariant under the map $T(t, x)=(c^{-\alpha}t, cx)$ for every $c>0$. This implies that $W^{(\alpha, \lambda)}(\theta^{-\alpha}, \theta)$ has the same law as $W^{(\alpha, \lambda)}(1,1)$ for every $\theta>0$. In the following, we denote $W^{(\alpha, \lambda)} = W^{(\alpha, \lambda)}(1,1)$. Now we show that the construction of $(\mathbf{Z}^{(\alpha, \lambda)}_n)_{n \geqslant 0}$ via Poisson point process $\mathbf{N}$ induces the self-decomposability of $W^{(\alpha, \lambda)}$.

\begin{proposition} \label{prop:sd}
The law of $W^{(\alpha, \lambda)}$ is self-decomposable, absolutely continuous and possesses a unimodal density function.
\end{proposition}

\begin{proof}
Fix some $s > 0$. Let $
\{ W_{(t,x)}: (t, x) \} 
$ be a family of i.i.d. copies of $W^{(\alpha, \lambda)}$ independent of $\mathbf{Z}_1^{(\alpha, \lambda)}$. Denote  
$$
A_s = \sum_{\substack{{\textrm{$(t,x)$ atom of $\mathbf{Z}_1^{(\alpha, \lambda)}$}}\\t<s}} e^{-t-x} W_{(t,x)}.
$$
Then we have the following identity in law
\begin{align*} 
W^{(\alpha, \lambda)} &\stackrel{(\textrm{d})}{=} \sum_{\substack{{\textrm{$(t,x)$ atom of $\mathbf{Z}_1^{(\alpha, \lambda)}$}}\\t\geqslant s}} e^{-t-x} W_{(t,x)}+A_s.
\end{align*}
Plainly, we also have
\begin{align} \label{eq:sd1}
W^{(\alpha, \lambda)} &\stackrel{(\textrm{d})}{=} e^{-s}\sum_{\substack{{\textrm{$(t,x)$ atom of $\mathbf{Z}_1^{(\alpha, \lambda)}$}}\\t\geqslant s}} e^{-(t-s)-x} W_{(t,x)}+A_s.
\end{align}
By the construction (\ref{eq:Poisson}) of $\mathbf{Z}_1^{(\alpha, \lambda)}$ via point process $\mathbf{N}$ and the fact that the law of Poisson point process $\mathbf{N}$ is invariant under mapping $(t, x) \mapsto (t-s, x)$, 
\begin{align} \label{eq:sd2}
W^{(\alpha, \lambda)} \stackrel{(\textrm{d})}{=} \sum_{\substack{{\textrm{$(t,x)$ atom of $\mathbf{Z}_1^{(\alpha, \lambda)}$}}\\t\geqslant 0}} e^{-t-x} W_{(t,x)} \stackrel{(\textrm{d})}{=}  \sum_{\substack{{\textrm{$(t,x)$ atom of $\mathbf{Z}_1^{(\alpha, \lambda)}$}}\\t\geqslant s}} e^{-(t-s)-x} W_{(t,x)}.
\end{align}
Combining distributional equations (\ref{eq:sd1}) and (\ref{eq:sd2}) gives us that
\begin{align*}
W^{(\alpha, \lambda)} \stackrel{(\textrm{d})}{=} e^{-s} W^{(\alpha, \lambda)} +A_s.
\end{align*}
This verifies that the law of $W^{(\alpha, \lambda)}$ is self-decomposable. The rest of the proof can be completed by Example 27.8 and Theorem 53.1 in Sato \cite{Sato1999}. 
\end{proof}

\begin{remark} \label{remark:infinitelydivisable}
Since self-decomposable laws are also infinitely divisible (Proposition 15.5 in \cite{Sato1999}), the laws of $W^{(\alpha, \lambda)}$ and $W$ are infinitely divisible on $\mathbb{R}_+$.
\end{remark}

Finally we can upgrade the results in Theorem \ref{theorem:simple} and give a precise description of asymptotic behaviors of the density function $f$ of $W$.

\begin{proof}[Proof of Theorem \ref{prop:density}] The first part of this theorem follows directly from Proposition \ref{prop:sd}. To prove the second part concerning the asymptotic behaviors of density $f$, first recall that we have shown in Theorem \ref{theorem:simple}(ii) that
$$
- \log \mathbb{P} ( W > x ) \sim r^* x, \quad x \rightarrow \infty.
$$
Since $f$ is non-increasing on some neighbourhood of infinity, we can apply Theorem 4.12.10(i) of Bingham et al. \cite{BGT87} to conclude that
$$
- \log f(x) \sim r^{*}x, \quad x \rightarrow \infty.
$$
On the other hand, from Theorem \ref{theorem:simple}(iii), 
$$
-\log \mathbb{P}(W<x) \sim \frac{1}{2} (\log x)^2, \quad x \rightarrow 0+.
$$
Since $f$ is non-decreasing on some neighbourhood of $0$, for $x$ small enough,
$$
-\log \left(xf(x)\right) \leq  -\log \int_0^x f(u) du \leq -\log \mathbb{P}(W<x)= \frac{1}{2}(\log x)^2 (1+o(1)).
$$
This gives us the upper bound that
$$
\limsup_{x \rightarrow 0+} \frac{-\log f(x)}{\frac{1}{2} (\log x)^2} \leq 1.
$$
For any $c>1$ and $x$ small enough, we also have
$$
-\log\left((c-1)xf(x)\right) \geqslant -\log \int_x^{cx} f(u)du \geqslant -\log \int_0^{cx} f(u)du = \frac{1}{2} \left(\log(cx) \right)^2 (1+o(1)).
$$
Dividing both sides of the above by $\frac{1}{2}(\log x)^2$ and sending $x \rightarrow 0+$ gives us the desired lower bound, which completes the proof.
\end{proof}

\begin{remark} \label{rk:Yule}
Consider standard Yule process $(Y_t)_{t \geqslant 0}$ with additive martingale $(M_t)_{t \geqslant 0}$ defined by
$$
M_t = \frac{Y_t}{\mathbb{E}(Y_t)}.
$$
It is known that $(M_t)_{t \geqslant 0}$ converges to non-degenerate terminal value $M$. Now we adapt our proof of Theorem \ref{theorem:simple} to show that $M$ follows the standard exponential equation. 

Consider point process $\tilde{\mathbf{Z}}$ whose atoms are given by the locations of the first generation of the Yule process $(Y_t)_{t\geqslant 0}$. Note that $\tilde{\mathbf{Z}}$ is a Poisson point process on $\mathbb{R}_+$ with intensity measure given by the Lebesgue measure on $\mathbb{R}_+$. Let $(\tilde{\mathbf{Z}}_n)_{n\geqslant 0}$ be a branching random walk on $\mathbb{R}_+$ started from the origin such that the first step $\tilde{\mathbf{Z}}_1$ is given by $\tilde{\mathbf{Z}}$. The additive martingale $(\tilde W_n)_{n \geqslant 0}$ given by
$$
\tilde W_n = \int_0^{\infty} e^{-x} \tilde{\mathbf{Z}}_n(dx)
$$
converges almost surely and in $L^1$ to its terminal value $\tilde W$. Similar to Proposition \ref{prop:coincide}, we can show in a similar procedure that 
$
M = \tilde W
$
almost surely. 

Now let $\tilde \phi$ be the Laplace transform of $\tilde W$. Since $\tilde W$ is a fixed point of smoothing transform and $\tilde{\mathbf{Z}}_1$ is a Poisson point process, we have 
$$
\log \tilde \phi(r) =\log  \mathbb{E} \left( \int_0^{\infty}  \tilde \phi(re^{-x}) \tilde{\mathbf{Z}}_1(dx) \right) = \int_0^{\infty} \left(\tilde \phi(re^{-x}) -1 \right) dx.
$$
Together with the initial condition that $(\tilde \phi)'(0)=1$, the above equation has a unique solution given by
$$
\tilde \phi (r) = \frac{1}{1+r}.
$$
From this, we conclude that terminal value $M$ follows standard exponential equation.
\end{remark}


\printbibliography

\end{document}